\documentclass[english]{elsarticle}
\usepackage[T1]{fontenc}
\usepackage[latin9]{inputenc}
\usepackage{color}
\usepackage{array}
\usepackage{float}
\usepackage{amsmath}
\usepackage{amsthm}
\usepackage{amssymb}

\makeatletter

\providecommand{\tabularnewline}{\\}

\theoremstyle{plain}
\newtheorem{thm}{\protect\theoremname}
  \theoremstyle{plain}

\usepackage{babel}
\usepackage{listings}
\usepackage[labelfont=bf,labelsep=space]{caption}

\renewcommand{\fnum@figure}{\textbf{Fig.~\thefigure}}

\usepackage{babel}
\usepackage{listings}
\addto\captionsamerican{}
\addto\captionsenglish{}

\usepackage{babel}
\usepackage{listings}
\addto\captionsamerican{}
\addto\captionsenglish{}

\usepackage{babel}
\usepackage{listings}
\addto\captionsamerican{}
\addto\captionsenglish{}

\usepackage{babel}
\providecommand{\lemmaname}{Lemma}
\providecommand{\theoremname}{Theorem}

\makeatother

\usepackage{babel}
  \providecommand{\lemmaname}{Lemma}
\providecommand{\theoremname}{Theorem}

\begin{document}
\begin{frontmatter}{}

\title{What is the meaning of the graph energy after all?}

\author{Ernesto Estrada and Michele Benzi}

\address{Department of Mathematics \& Statistics, University of Strathclyde,\\
26 Richmond Street, Glasgow G11XQ, UK  

\medskip

Department of Mathematics and
Computer Science, Emory University,\\
 Atlanta, GA 30322, USA}

\begin{abstract}
For a simple graph $G=\left(V,E\right)$ with eigenvalues of the adjacency
matrix $\lambda_{1}\geq\lambda_{2}\geq\cdots\geq\lambda_{n}$, the
energy of the graph is defined by $E\left(G\right)=\sum_{j=1}^{n}\left|\lambda_{j}\right|$.
Myriads of papers have been published in the mathematical and chemistry
literature about properties of this graph invariant due to its connection
with the energy of (bipartite) conjugated molecules. However, a structural
interpretation of this concept in terms of the contributions of even
and odd walks, and consequently on the contribution of subgraphs,
is not yet known. Here, we find such interpretation and prove that
the (adjacency) energy of any graph (bipartite or not) is a weighted
sum of the traces of even powers of the adjacency matrix. We then
use such result to find bounds for the energy in terms of subgraphs
contributing to it. The new bounds are studied for some specific simple
graphs, such as cycles and fullerenes. We observe that including contributions
from subgraphs of sizes not bigger than 6 improves some of the best
known bounds for the energy, and more importantly gives insights about
the contributions of specific subgraphs to the energy of these graphs. 
\end{abstract}

\end{frontmatter}{}


\section{Introduction}

The concept of graph energy arose in the context of the study of conjugated
hydrocarbons using a tight-binding method known in chemistry as the
\textit{Hückel molecular orbital} (HMO) method (see for instance \citep{coulson1978huckel,dewar1969molecular}).
In this context, the total energy of a conjugated molecule $M$ is
defined by

\begin{equation}
E\left(M\right)=2\sum_{j:\lambda_{j}>0}\lambda_{j}=\sum_{j=1}^{n}\left|\lambda_{j}\right|,\label{eq:HMO energy}
\end{equation}
where the last equality is a consequence of the fact that such conjugated
molecules can be represented by bipartite graphs, thus the spectra
of their adjacency matrices are symmetric. Here, $\lambda_{1}\geq\lambda_{2}\geq\cdots\geq\lambda_{n}$
are the eigenvalues of the adjacency matrix of the (molecular) graph\textemdash typically
a simple, connected, undirected graph. It should be remarked that
such energy is given in units of a parameter known as $\beta$ which
has a negative value.

This concept was then generalized to any graph\textemdash not necessarily
bipartite\textemdash by Ivan Gutman, who named it the \textit{graph
energy} \citep{gutman1978energy}. Then, for a simple, undirected graph
$G=\left(V,E\right)$, the energy is defined as

\begin{equation}
E\left(G\right)=\sum_{j=1}^{n}\left|\lambda_{j}\right|.
\end{equation}

A myriad of papers and a couple of monographs have been written about
the graph energy \citep{gutman2001energy,Gutman2016energy,li2012graph}.
The monograph \citep{li2012graph} is an excellent compilation of
results, historical background and methodological approaches that
may serve as a guide to the reader who wants to get deeper into this
field. The concept has been generalized to other matrices apart from
the adjacency one \citep{nikiforov2007energy} (see also the corresponding
Chapters of the monographs \citep{Gutman2016energy,li2012graph}),
and many bounds and extremal properties have been reported for these
graph/matrix energies. Many researchers claim in their papers that
they are studying the graph energy because of the chemical implications
of that quantity. As soon as this concept is extended to non-bipartite
graphs, however, it completely loses all its chemical and physical
meaning. Nevertheless, as a graph invariant, the graph energy can
bring important structural information about the graph. But, the problem
is to know what exactly the graph energy means in terms of the structure
of a graph. Thus, after nearly 40 years of research on graph energy, what
is it after all?

Here we provide a structural interpretation of the graph energy using
the concept of matrix function (see next Section for formal definitions).
In particular, we prove that the graph energy is given by the sum
of the traces of the even powers of the adjacency matrix weighted
in a specific way. Using this new representation we find new bounds
for the energy as sums of contributions of subgraphs. Consequently,
armed with this structural interpretation the graph energy can now
be used in the general context of structural graph theory or even
to study some real-world graphs.

\section{Preliminaries}

We consider here simple, undirected, connected graphs $G=\left(V,E\right)$,
without multiple edges or self-loops. The adjacency matrix $A$ of
$G$ is then a square, symmetric matrix with spectral decomposition
$A=V\Lambda V^{T}$, where 
\[
V=\left[\begin{array}{ccc}
\vec{\psi}_{1} & \ldots & \vec{\psi}_{n}\end{array}\right]
\]
is the matrix of orthonormalized eigenvectors $\vec{\psi}_{j}$ associated
with the eigenvalues $\lambda_{j}$, and $\Lambda={\rm diag}\left(\lambda_{1},\ldots,\lambda_{n}\right)$.
If $f$ is a scalar function defined on the spectrum of $A$ we can
define a function of the matrix $A$, $f\left(A\right)$, by means
of 
\begin{equation}
f\left(A\right)=Vf\left(\Lambda\right)V^{T},
\end{equation}
where $f(\Lambda)={\rm diag}\left(f(\lambda_{1}),\ldots,f(\lambda_{n})\right)$.
For example, for any symmetric positive semidefinite matrix $A$ we
can define its (positive semidefinite) square root by means of $S=\sqrt{A}=V\sqrt{\Lambda}V^{T}$.
This is the only symmetric positive semidefinite matrix with the property
that $S^{2}=A$.

We observe that if $f$ is defined by a power series expansion of
the form 
\[
f(x)=\sum_{k=0}^{\infty}a_{k}x^{k}
\]
such that the series converges on an open disk containing the $\lambda_{j}$,
then the above definition is equivalent to 
\begin{equation}
f\left(A\right)=\sum_{k=0}^{\infty}a_{k}A^{k}.
\end{equation}

For further information on matrix functions, the reader is referred
to \citep{higham2008functions}.

\section{Main result}

The main result of this work is the finding that the (adjacency) graph
energy of any graph can be obtained as a weighted sum of even powers
of the adjacency matrix. First, we observe that 
\begin{equation}
E\left(G\right)={\rm tr}\left|A\right|,
\end{equation}
where $\left|A\right|=V|\Lambda|V^{T}$ stands for the absolute value
matrix function of $A$. Then, we have the following result. 
\begin{thm}
The energy of a graph is given by

\begin{equation}
E\left(G\right)=\lambda_{1}{\rm tr}\sum_{k=0}^{\infty}\left(\begin{array}{c}
\tfrac{1}{2}\\
k
\end{array}\right)\sum_{\ell=0}^{k}\left(-1\right)^{\ell}\left(\begin{array}{c}
k\\
\ell
\end{array}\right)\left(\dfrac{A}{\lambda_{1}}\right)^{2\ell}.\label{eq:main_result}
\end{equation}
\end{thm}
\begin{proof}
We start by recalling that every symmetric positive semidefinite matrix
has a unique positive semidefinite square root. Then, we have that

\begin{equation}
\left|A\right|=\sqrt{A^{2}}.
\end{equation}

We now expand the square root in a power series in $A^{2}$. Let $\lambda_{1}>0$
be the largest eigenvalue of $A$. We note in passing that since $G$
is connected, $\lambda_{1}$ is simple. Then, $\tfrac{A}{\lambda_{1}}$
has spectral radius 1, and the matrix $B=\left(\lambda_{1}^{-1}A\right)^{2}-I$
has all its eigenvalues in the interval $\left[-1,0\right].$ Hence, $B$
is positive semidefinite and has spectral radius 1. Let us write 
\begin{equation}
\left|A\right|=\sqrt{A^{2}}=\lambda_{1}\sqrt{\left(\dfrac{A}{\lambda_{1}}\right)^{2}}=\lambda_{1}\sqrt{I+\left(\left(\dfrac{A}{\lambda_{1}}\right)^{2}-I\right)}=\lambda_{1}\left(I+B\right)^{\tfrac{1}{2}}.
\end{equation}

Recall now the following special case of the binomial theorem:

\begin{equation}
\sqrt{1+x}=\left(1+x\right)^{\tfrac{1}{2}}=\sum_{k=0}^{\infty}\left(\begin{array}{c}
\tfrac{1}{2}\\
k
\end{array}\right)x^{k},
\end{equation}
where the series converges for all $x\in\left[-1,1\right]$, and 
\begin{equation}
\left(\begin{array}{c}
\alpha\\
k
\end{array}\right):=\dfrac{\alpha\left(\alpha-1\right)\cdots\left(\alpha-k+1\right)}{k!}
\end{equation}
for any real $\alpha$ (here $\alpha=\tfrac{1}{2}$). Therefore we
can write 
\begin{equation}
\left|A\right|=\lambda_{1}\sum_{k=0}^{\infty}\left(\begin{array}{c}
\tfrac{1}{2}\\
k
\end{array}\right)B^{k}=\lambda_{1}\sum_{k=0}^{\infty}\left(\begin{array}{c}
\tfrac{1}{2}\\
k
\end{array}\right)\left(\left(\dfrac{A}{\lambda_{1}}\right)^{2}-I\right)^{k},
\end{equation}
which readily gives the desired result. 
\end{proof}
If we consider the first few terms of the expansion for $E$ we have
\begin{equation}
E={\rm tr}\left|A\right|=\lambda_{1}{\rm tr}\left[I+\dfrac{1}{2}\left(\dfrac{A^{2}}{\lambda_{1}^{2}}-I\right)-\dfrac{1}{2\cdot4}\left(\dfrac{A^{2}}{\lambda_{1}^{2}}-I\right)^{2}+\dfrac{1\cdot3}{2\cdot4\cdot6}\left(\dfrac{A^{2}}{\lambda_{1}^{2}}-I\right)^{3}-\cdots\right],
\end{equation}
which clearly indicates that the energy of a graph only depends on
even powers of the adjacency matrix of the corresponding graph. That
is, 
\begin{equation}
E={\rm tr}\left|A\right|=\lambda_{1}\left[\sum_{k=0}^{\infty}\left(\begin{array}{c}
2k\\
k
\end{array}\right)\dfrac{\left(-1\right)^{k+1}}{2^{2k}\left(2k-1\right)}{\rm tr}\left(\dfrac{A^{2}}{\lambda_{1}^{2}}-I\right)^{k}\right].
\end{equation}

\section{Further developments}

Here we use the main result in the previous section to obtain some
upper bounds for the energy of a graph. Our goal is not to give very
sharp bounds but to derive some that allow us to interpret the structural
meaning of the graph energy, with special emphasis on the molecular
energy in the HMO method. Recall that we have set 
\begin{equation}
B=\left(\dfrac{A^{2}}{\lambda_{1}^{2}}-I\right),
\end{equation}
and that $B$ is a negative semidefinite matrix (with spectrum in
$[-1,0]$). Furthermore, 
\begin{equation}
B_{ii}=\left(\dfrac{A^{2}}{\lambda_{1}^{2}}-I\right)_{ii}=\frac{k_{i}}{\lambda_{1}^{2}}-1,\label{eq:diag}
\end{equation}
where $k_{i}$ is the degree of the corresponding vertex. Clearly,
these diagonal terms are all negative. We now prove the following
result.
\begin{thm}
Let $G$ be a graph with $n$ nodes and $m$ edges. Then, 
\end{thm}
\begin{equation}
E\left(G\right)\leq\left(\dfrac{\lambda_{1}}{2}\right)n+\left(\dfrac{1}{\lambda_{1}}\right)m.\label{eq:Bound1}
\end{equation}

\begin{proof}
It is easy to see that 
\begin{equation}
E\left(G\right)\leq\lambda_{1}{\rm tr}I=\lambda_{1}n,
\end{equation}
and that 
\begin{equation}
E\left(G\right)=\lambda_{1}\left[{\rm tr}I+\dfrac{1}{2}{\rm tr}\left(\dfrac{A^{2}}{\lambda_{1}^{2}}-I\right)-\left|\left(\dfrac{1}{8}{\rm tr}\left(\dfrac{A^{2}}{\lambda_{1}^{2}}-I\right)^{2}+\dfrac{3}{48}{\rm tr}\left(\dfrac{A^{2}}{\lambda_{1}^{2}}-I\right)^{3}+\cdots\right)\right|\right],
\end{equation}
which implies that 
\begin{equation}
E\left(G\right)\leq\lambda_{1}\left[{\rm tr}I+\dfrac{1}{2}{\rm tr}\left(\dfrac{A^{2}}{\lambda_{1}^{2}}-I\right)\right]=\lambda_{1}n+\left(\dfrac{1}{\lambda_{1}}\right)\sum_{i=1}^nk_{i}-\left(\dfrac{\lambda_{1}}{2}\right)n.
\end{equation}
The result is now an immediate consequence of (\ref{eq:diag}). 
\end{proof}
It can be easily verified that 
\begin{equation}
E\left(G\right)\leq\sqrt{2mn}\leq\left(\dfrac{\lambda_{1}}{2}\right)n+\left(\dfrac{1}{\lambda_{1}}\right)m,
\end{equation}
where the first bound is the well-known McClelland one \citep{mcclelland1971properties}.
However, we can systematically improve the bound found above by using
the same approach used for its proof, as we will show in the next
two results. 
\begin{thm}
Let $G$ be a graph with $n$ nodes, $m$ edges, $P_{3}$ paths of
three vertices and $C_{4}$ cycles of length four. Then, 
\end{thm}
\begin{equation}
E\left(G\right)\leq\left(\dfrac{3\lambda_{1}}{8}\right)n+\left(\dfrac{6\lambda_{1}^{2}-1}{4\lambda_{1}^{3}}\right)m-\left(\dfrac{1}{2\lambda_{1}^{3}}\right)P_{3}-\left(\dfrac{1}{\lambda_{1}^{3}}\right)C_{4}.\label{eq:Bound2}
\end{equation}

\begin{proof}
It is easy to see that 
\begin{align}
{\rm tr}\left(B^{2}\right) & ={\rm tr}\left(\dfrac{A^{2}}{\lambda_{1}^{2}}-I\right)^{2}\nonumber \\
 & =\dfrac{1}{\lambda_{1}^{4}}{\rm tr}A^{4}-\dfrac{2}{\lambda_{1}^{2}}{\rm tr}A^{2}+{\rm tr}I.
\end{align}
We can now obtain the traces of $A^{4}$ and of $A^{2}$ in terms
of the subgraphs contributing to them. It is known (see, e.g., \citep[page 137]{estradaknightbook})
that 
\begin{equation}
{\rm tr}A^{4}=2m+4P_{3}+8C_{4},
\end{equation}
and of course 
\begin{equation}
{\rm tr}A^{2}=2m.
\end{equation}
Then, by plugging these two formulas into the expression for $\dfrac{\lambda_{1}}{8}{\rm tr}\left(B^{2}\right)$
we get 
\begin{equation}
\dfrac{\lambda_{1}}{8}{\rm tr}\left(B^{2}\right)=\left(\dfrac{\lambda_{1}}{8}\right)n+\left(\dfrac{1-4\lambda_{1}^{2}}{4\lambda_{1}^{3}}\right)m+\left(\dfrac{1}{2\lambda_{1}^{3}}\right)P_{3}-\left(\dfrac{1}{\lambda_{1}^{3}}\right)C_{4}.
\end{equation}
Finally, by taking

\begin{equation}
E\left(G\right)\leq\left(\dfrac{\lambda_{1}}{2}\right)n+\left(\dfrac{1}{\lambda_{1}}\right)-\dfrac{\lambda_{1}}{8}{\rm tr}\left(B^{2}\right)
\end{equation}
we obtain the result. 
\end{proof}
An important feature of this bound is that it clearly agrees with
the chemical intuition. For instance, it is well known that conjugated
$C_{4}$ cycles destabilize a molecule, due to their increase in the
molecular energy. We recall that the energy $E\left(G\right)$ for
a molecule is given in units of $\beta$ which is negative. Thus,
both terms $P_{3}$ and $C_{4}$ make contributions that increase
the total energy of a molecule.

We can improve the previous bound by using a similar approach. 
\begin{thm}
Let $G$ be a graph. Then,

\begin{align}
E\left(G\right) & \leq\left(\dfrac{5\lambda_{1}}{16}\right)n+\left(\dfrac{30\lambda_{1}^{4}-12\lambda_{1}^{2}+1}{16\lambda_{1}^{5}}\right)m-\left(\dfrac{5\lambda_{1}^{2}-3}{4\lambda_{1}^{5}}\right)P_{3}-\left(\dfrac{7\lambda_{1}^{2}-12}{4\lambda_{1}^{5}}\right)C_{4}\nonumber \\
 & +\left(\dfrac{3}{2\lambda_{1}^{5}}\right)C_{3}+\left(\dfrac{3}{4\lambda_{1}^{5}}\right)P_{4}+\left(\dfrac{3}{2\lambda_{1}^{5}}\right)S_{1,3}+\left(\dfrac{9}{2\lambda_{1}^{5}}\right)D_{4}+\left(\dfrac{3}{2\lambda_{1}^{5}}\right)F\nonumber \\
 & +\left(\dfrac{3}{2\lambda_{1}^{5}}\right)H+\left(\dfrac{3}{2\lambda_{1}^{5}}\right)C_{6}.\label{eq:Bound_3}
\end{align}
where $C_{n}$ and $P_{n}$ represent cycles and paths with $n$ vertices,
$S_{1,3}$ is the star subgraph with one central vertex and 3 pendant
ones, and $D_{4}$ is the diamond graph, i.e., a graph consisting
of $C_{4}$ in which an edge is added to one pair of nonadjacenct
vertices, $F$ is a subgraph consisting of one square with a pendant
node and $H$ is a subgraph consisting of two triangles sharing a
common node. 
\end{thm}
\begin{proof}
It is easy to see that 
\begin{align}
{\rm tr}\left(B^{3}\right) & ={\rm tr}\left(\dfrac{A^{2}}{\lambda_{1}^{2}}-I\right)^{3}\nonumber \\
 & =\dfrac{1}{\lambda_{1}^{6}}{\rm tr}A^{6}-\dfrac{3}{\lambda_{1}^{4}}{\rm tr}A^{4}+\dfrac{3}{\lambda_{1}^{2}}{\rm tr}A^{2}-{\rm tr}I.
\end{align}
The expression for ${\rm tr}A^{6}$ in terms of subgraphs (see \citep[page 139]{estradaknightbook})
is given by:

\begin{equation}
{\rm tr}A^{6}=2m+12P_{3}+24C_{3}+48C_{4}+12S_{1,3}+6P_{4}+36D_{4}+12F+24H+12C_{6}.
\end{equation}
Then, by plugging these two formulas into the expression for $\dfrac{\lambda_{1}}{16}{\rm tr}\left(B^{3}\right)$
we get 
\begin{align}
\dfrac{\lambda_{1}}{16}{\rm tr}\left(B^{3}\right) & =\left(\dfrac{1-6\lambda_{1}^{2}-6\lambda_{1}^{4}}{16\lambda_{1}^{5}}\right)m+\left(\dfrac{3-3\lambda_{1}^{2}}{4\lambda_{1}^{5}}\right)P_{3}-\left(\dfrac{12-3\lambda_{1}^{2}}{4\lambda_{1}^{5}}\right)C_{4}\nonumber \\
 & +\left(\dfrac{3}{4\lambda_{1}^{5}}\right)C_{3}+\left(\dfrac{3}{8\lambda_{1}^{5}}\right)P_{4}+\left(\dfrac{3}{4\lambda_{1}^{5}}\right)S_{1,3}+\left(\dfrac{9}{4\lambda_{1}^{5}}\right)D_{4}+\left(\dfrac{3}{4\lambda_{1}^{5}}\right)F\nonumber \\
 & +\left(\dfrac{3}{2\lambda_{1}^{5}}\right)H+\left(\dfrac{3}{4\lambda_{1}^{5}}\right)C_{6}.
\end{align}
Then by taking 
\begin{equation}
E\left(G\right)\leq\left(\dfrac{3\lambda_{1}}{8}\right)n+\left(\dfrac{6\lambda_{1}^{2}-1}{4\lambda_{1}^{3}}\right)m-\left(\dfrac{1}{2\lambda_{1}^{3}}\right)P_{3}-\left(\dfrac{1}{\lambda_{1}^{3}}\right)C_{4}+\dfrac{\lambda_{1}}{16}{\rm tr}\left(B^{3}\right),
\end{equation}
we get the final result. 
\end{proof}
This result clearly indicates that subgraphs like $C_{6}$ contribute
to decreasing the energy of a graph. In molecular systems this is
an important result due to the well-know fact that benzenoid molecules,
which are constructed on the basis of fusing together $C_{6}$ fragments,
are very stable. However, the result shows also other fragments which
contribute to the stabilization of conjugated molecular systems as
the ones studied in the HMO context. This includes the fragment $P_{4}$
which obviously corresponds to the butadiene fragment and which is
easily recognizable as a stabilizing fragment. Other fragments appear
here in a more unexpected way, such as $C_{3}$, $D_{4}$, $F$ and
$H$.

\section{Bounding individual fragment contributions}

One important consequence of the findings of this paper is that we
can obtain bounds for the contribution of individual subgraphs to
the total graph energy. For instance, suppose that we are interested
in knowing how the subgraph $C_{8}$ contributes to $E\left(G\right)$.
Then, we can do the following. We first identify the first spectral
moment of the matrix $B$ in which $C_{8}$ contributes. That is,
$C_{8}$ appears for the first time in the term $\dfrac{{\rm tr}A^{8}}{\lambda_{1}^{8}}$
of $B^{4}$. Thus, let $\eta_{G}\left(C_{8}\right)$ be the contribution
of the cycle of 8 nodes to the total energy of a graph $G$ and let
$\eta_{8}\left(C_{8}\right)$ be the contribution of $C_{8}$ to the
8th spectral moment of $A$, i.e., $\eta_{8}\left(C_{8}\right)=16$.
Then,

\begin{equation}
\eta_{G}\left(C_{8}\right)\leq\lambda_{1}\left(\begin{array}{c}
2\cdot4\\
4
\end{array}\right)\dfrac{\left(-1\right)^{4+1}}{2^{2\cdot4}\left(2\cdot4-1\right)}\left(\dfrac{\eta_{8}\left(C_{8}\right)}{\lambda_{1}^{8}}\right)=-\left(\dfrac{5}{8\lambda_{1}^{7}}\right).
\end{equation}

The negative sign indicates that an octacycle increases the energy
of the graph. In the case of molecules treated under the HMO scheme,
it is well known that cycles with $4n$ atoms destabilize the molecule,
which is exactly the result obtained here. Of course, to improve this
bound it is necessary to find the contributions of this subgraph to
higher moments of the matrix $B$, but in this way it is at least
possible to obtain bounds and to analyze the chemical impact of such
fragments in a molecule.

\section{Numerical results}

In this section we show some numerical results on the different bounds
obtained in this paper for simple graphs of importance in chemistry.
The goal of these bound is not to obtain good approximations of the graph
energy for these graphs. Indeed, the direct calculation of the energy
for these graphs is easier than the calculation of the bounds. Our
goal is to show how the incorporation of certain subgraphs into the
bounds improves them and providing a structural interpretation of
the graph energy for such graphs. In all cases we compare our bounds
with the one of McClelland \citep{mcclelland1971properties}, which
is simple and remarkably good in approaching the graph energy. 

First we study a series of cycle graphs $C_{n}$ for $3\leq n\leq10$.
In Table 1 we give the values of the energy and the results
of bounding it with the three bounds obtained here as well as by McClelland's
one. As can be seen even for such simple graphs the current approach
needs to incorporate terms coming from the ${\rm tr}A^{6}$ in order
to improve McClelland's bound. In this type of graphs, a few of the
subgraphs contributing to (\ref{eq:Bound_3}) are not present, e.g.,
$S_{1,3}$, $D_{4}$, $F$, $H,$ and some of the others only appear in
specific graphs, such as $C_{3}$, $C_{4}$, and $C_{6}$. Thus, the
main improvement in this bound in relation to the other two comes
from the better account of the contributions of $n$, $m$ and $P_{3}$
and the newly introduced contribution of $P_{4}$.

\begin{center}
\begin{table}[H]
\begin{centering}
\begin{tabular}{|>{\centering}p{1cm}|>{\centering}p{1.2cm}|>{\centering}p{1.2cm}|>{\centering}p{1.2cm}|>{\centering}p{1.2cm}|>{\centering}p{1.2cm}|}
\hline 
$n$  & $E\left(G\right)$  & $\sqrt{2mn}$  & (\ref{eq:Bound1})  & (\ref{eq:Bound2})  & (\ref{eq:Bound_3})\tabularnewline
\hline 
\hline 
3  & 4.000  & 4.243  & 4.5  & 4.219  & 4.113\tabularnewline
\hline 
4  & 4.000  & 5.657  & 6  & 5.500  & 5.250\tabularnewline
\hline 
5  & 6.472  & 7.071  & 7.5  & 7.031  & 6.836\tabularnewline
\hline 
6  & 8.000  & 8.485  & 9  & 8.438  & 8.227\tabularnewline
\hline 
7  & 8.988  & 9.899  & 10.5  & 9.844  & 9.570\tabularnewline
\hline 
8  & 9.657  & 11.314  & 12  & 11.250  & 10.938\tabularnewline
\hline 
9  & 11.517  & 12.728  & 13.5  & 12.656  & 12.305\tabularnewline
\hline 
10  & 12.944  & 14.142  & 15  & 14.062  & 13.672\tabularnewline
\hline 
\end{tabular}
\par\end{centering}
{\caption{Values of the energy $E\left(G\right)$ and their estimation using
McClelland's bound $\sqrt{2mn}$ as well as the bounds obtained in
this work for the cycle graphs $C_{n}$ with different number of nodes. }
}
\label{Table1}
\end{table}
\par\end{center}

As a second example we study a series of fullerene graphs having from
20 to 540 nodes. The results are given in Table 2. Here again,
it is necessary to go beyond the contribution of ${\rm tr}A^{4}$
to make improvements over McClelland's bound. Here the main contributions
to this improvement are made by $n$, $m$, $P_{3}$, $P_{4}$, $S_{1,3}$,
and $C_{6}$. Notice, that the contributions of $C_{5}$ are only
captured after the consideration of ${\rm tr}A^{10}$, which is not
studied here. It can be then said that the energy of fullerenes is
bounded by the following specific expression:

\begin{align}
E\left(G\right) & \leq\left(\dfrac{5\lambda_{1}}{16}\right)n+\left(\dfrac{30\lambda_{1}^{4}-12\lambda_{1}^{2}+1}{16\lambda_{1}^{5}}\right)m-\left(\dfrac{5\lambda_{1}^{2}-3}{4\lambda_{1}^{5}}\right)P_{3} \nonumber \\
 & +\left(\dfrac{3}{4\lambda_{1}^{5}}\right)P_{4}+\left(\dfrac{3}{2\lambda_{1}^{5}}\right)S_{1,3}+\left(\dfrac{3}{2\lambda_{1}^{5}}\right)C_{6}.
\end{align}

\begin{center}
\begin{table}[t]
\begin{centering}
\begin{tabular}{|>{\centering}p{1cm}|>{\centering}p{1.4cm}|>{\centering}p{1.4cm}|>{\centering}p{1.4cm}|>{\centering}p{1.4cm}|>{\centering}p{1.4cm}|}
\hline 
$n$  & $E\left(G\right)$  & $\sqrt{2mn}$  & (\ref{eq:Bound1})  & (\ref{eq:Bound2})  & (\ref{eq:Bound_3})\tabularnewline
\hline 
\hline 
20  & 29.416  & 34.641  & 40  & 36.111  & 34.4753\tabularnewline
\hline 
24  & 36.022  & 41.569  & 48  & 43.330  & 41.376\tabularnewline
\hline 
26  & 39.742  & 45.0333  & 52  & 46.944  & 44.827\tabularnewline
\hline 
28  & 43.107  & 48.497  & 56  & 50.555  & 48.278\tabularnewline
\hline 
30  & 45.704  & 51.962  & 60  & 54.167  & 51.728\tabularnewline
\hline 
32  & 49.150  & 55.425  & 64  & 57.778  & 55.179\tabularnewline
\hline 
36  & 55.244  & 62.350  & 72  & 65.000  & 62.080\tabularnewline
\hline 
50  & 77.579  & 86.602  & 100  & 90.278  & 86.235\tabularnewline
\hline 
60  & 93.162  & 103.923  & 120  & 108.333  & 103.488\tabularnewline
\hline 
76  & 118.326  & 131.636  & 152  & 137.222  & 131.093\tabularnewline
\hline 
80  & 121.617  & 135.1  & 156  & 140.833  & 134.543\tabularnewline
\hline 
180  & 282.066  & 311.769  & 360  & 325.000  & 310.525\tabularnewline
\hline 
240  & 376.535  & 415.692  & 480  & 433.333  & 414.043\tabularnewline
\hline 
320  & 502.831  & 554.256  & 640  & 577.778  & 552.067\tabularnewline
\hline 
540  & 848.924  & 935.307  & 1080  & 975.000  & 931.636\tabularnewline
\hline 
\end{tabular}
\par\end{centering}
{\caption{Values of the energy $E\left(G\right)$ and their estimation using
McClelland's bound $\sqrt{2mn}$ as well as the bounds obtained in
this work for the fullerene graphs with different number of nodes. }
}
 \label{Table2}
\end{table}
\par\end{center}

\section{Conclusions}

The main conclusion of this work is that the graph energy is a weighted
sum of the traces of even powers of the adjacency matrix. The potential
advantages of this finding is that new techniques can be designed
to bound the energy of graphs, in which the specific contribution
of subgraphs can be obtained. This is of great importance in chemistry
where the search for additive rules for molecular properties is a
golden rule for understanding such properties in structural terms.
Finally, we hope that the new findings reported here might allow to
better understand certain properties of the graph energy in certain
families of graphs.

\vspace{0.2in}

\noindent \textbf{Acknowledgements}\\
 EE thanks the Royal Society of London for a Wolfson Merit Research
Award. The work of MB was supported by NSF grant DMS-1418889.

 \bibliographystyle{aps-nameyear}

\begin{thebibliography}{99}

\bibitem{coulson1978huckel}
{\sc C.~A.~Coulson, B.~O'Leary, and R.~B.~Mallion},
{\em H{\"u}ckel Theory for Organic Chemists},
Academic Press, London, 1978.

\bibitem{dewar1969molecular}
{\sc M.~J.~S.~Dewar and R.~Jones},
{\em The Molecular Orbital Theory of Organic Chemistry},
MCGraw--Hill, New Yourk, 1969.

\bibitem{estradaknightbook}
{\sc E.~Estrada and P.~A.~Knight},
{\em A First Course in Network Theory},
Oxford University Press, Oxford, UK, 2015.

\bibitem{gutman1978energy}
{\sc I.~Gutman},
{\rm The energy of a graph},
Ber. Math.-Stat. Sekt. Forschungszent. Graz,
{\bf 103}, pp.~1--22 (1978).

\bibitem{gutman2001energy}
{\sc I.~Gutman},
{\rm The energy of a graph: old and new results},
in {\em Algebraic Combinatorics and Applications},
Springer, Berlin and Heidelberg 2001, pp.~196--211.
 
\bibitem{Gutman2016energy} 
{\sc I.~Gutman and X.~Li},
{\em Energies of Graphs. Theory and Applications},
University of Kragujevac, Kragujevac, Serbia, 2016.

\bibitem{higham2008functions}
{\sc N.~J.~Higham},
{\em Functions of Matrices: Theory and Computation},
Society for Industrial and Applied Mathematics, Philadelphia,
2008.

\bibitem{li2012graph}
{\sc X.~Li, Y.~Shi, and I.~Gutman},
{\em Graph Energy}, 
Springer, New York, 2012.

\bibitem{mcclelland1971properties}
{\sc B.~McClelland},
{\rm Properties of the latent roots of a matrix: The estimation of pi-electron energies},
Journal of Chemical Physics, {\bf 54}, pp.~640--643 (1971).

\bibitem{nikiforov2007energy}
{\sc V.~Nikiforov},
{\rm The energy of graphs and matrices},
Journal of Mathematical Analysis and Applications,
{\bf 326}, pp.~1472--1475 (2007).









\end{thebibliography}

\end{document}